\documentclass[leqno]{article}
\usepackage{amsmath}
\usepackage[latin1]{inputenc}
\usepackage{amssymb}
\usepackage{xcolor}
\usepackage{amsthm}                                         
\usepackage{url}
\usepackage{graphicx}
\usepackage{epsfig}
\newtheorem{theorem}{\bf Theorem}

\newtheorem{lemma}[theorem]{\bf Lemma}

\newtheorem{proposition}[theorem]{\bf Proposition}
\newtheorem{definition}[theorem]{\bf Definition}

\newtheorem{remark}[theorem]{\bf Remark}

\numberwithin{equation}{section}
\numberwithin{theorem}{section}
\numberwithin{figure}{section}

\def\R{\mathbb{R}}

\def\R{\mathbb{R}}

\begin{document}
\renewcommand{\thefootnote}{}
\footnotetext{Research partially supported by Ministerio de Educaci\'on Grants No: MTM2013-43970-P  and partially supported by the Coordena\c{c}\~ao de Aperfei\c{c}oamento de Pessoal de N\'ivel Superior - Brasil (CAPES) - Finance Code 001} 

\title{Uniqueness of the $[\varphi,\vec{e}_{3}]$-catenary cylinders by their asymptotic behaviour} 
\author{$\text{A.L. Mart\'inez-Trivi\~no}^{1} \text{ and J.P. dos Santos}^{2}$}
\vspace{.1in}
\maketitle

{
\noindent $^1$Departamento de Geometr\'\i a y Topolog\'\i a, Universidad de Granada, E-18071, Granada, Spain\\ 
\noindent $^2$Departamento de Matem\'atica, Universidade de Bras\'\i lia, Bras\'\i lia-DF, Brazil \vspace*{2mm} \\ 
e-mails: aluismartinez@ugr.es, joaopsantos@unb.br}

\begin{abstract}
 We establish a uniqueness result for the $[\varphi,\vec{e}_{3}]$-catenary cylinders by their asymptotic behaviour. Well known examples of such cylinders are the grim reaper translating solitons for the mean curvature flow. For such solitons, F. Mart\'in, J. P\'erez-Garc\'ia, A. Savas-Halilaj and K. Smoczyk proved that, if $\Sigma$ is a properly embedded translating soliton with locally bounded genus, and $\mathcal{C}^{\infty}$-asymptotic to two vertical planes outside a cylinder, then $\Sigma$ must coincide with some grim reaper translating soliton. In this paper, applying the moving plane method of Alexandrov together with a strong maximum principle for elliptic operators, we increase the family of $[\varphi,\vec{e}_{3}]$-minimal graphs where these types of results hold under different assumption of asymptotic behaviour.
\end{abstract}
\vspace{0.2 cm}


\everymath={\displaystyle}

\section{Introduction}

Let $\varphi: \Omega \subset\mathbb{R}^3 \rightarrow\mathbb{R}$ be a smooth function and $\Omega$ an open subset of $\mathbb{R}^{3}$. An orientable immersion $\Sigma$  in $\Omega$ is called $\varphi$-minimal if and only if the mean curvature $H$ verifies the following equation
 \begin{equation}
\label{defphimin-original}
H=-\langle\overline{\nabla}\varphi, N\rangle,
\end{equation}
where $N$ is the Gauss map, $\overline{\nabla}$ denotes the usual gradient in $\mathbb{R}^{3}$ and $\langle\cdot,\cdot\rangle$ stands the usual Euclidean metric. From the work of T.Ilmanen \cite{Ilm94}, any $\varphi$-minimal surface can be viewed as a minimal surface in a conformal Riemannian 3-manifold $\Omega^{\varphi}:=(\Omega,\langle\cdot,\cdot\rangle^{\varphi})$ whose metric is defined for any $p \in\Omega^{\varphi}$ by

\begin{equation*}
\langle\cdot,\cdot\rangle^{\varphi}_{p}:=e^{\varphi (p)}\langle\cdot,\cdot\rangle_{p}.
\end{equation*}
Moreover, they can be viewed as critical points of the weighted volume functional,
\begin{equation*}
\mathcal{A}^{\varphi}(\Sigma)=\int_{\Sigma}\, e^{\varphi}\, d\Sigma,
\end{equation*}
where $d\Sigma$ is the induced volume element of $\Sigma$. 

In this paper we are interested in the case when the function $\varphi$ is invariant under a two-parameter group of translations. Up to rigid motions, if we write $p=(x_1, x_2, x_3)$ we can assume $\varphi$ depending only on $x_3$. Specifically, we will consider $\Omega$ as an open subset $\mathbb{R}^2 \times I$ and $\varphi : I \subset \mathbb{R} \rightarrow \mathbb{R}$ a smooth function. In this case, the equation \eqref{defphimin-original} is written as
\begin{equation}
H = -\dot{\varphi} \langle \vec{e}_3, N \rangle, \label{defphimin}
\end{equation}
where $\{\vec{e}_{i}\}_{i=1,2,3}$ is the usual orthonormal frame of $\mathbb{R}^{3}$ and $(\,\dot{ }\,)$ stands for the derivative with respect to $x_3$. A surface $\Sigma \subset \Omega$ which satisfies the equation \eqref{defphimin} will be called $[\varphi,\vec{e}_{3}]$\emph{-minimal surface}.

Some particular cases of these surfaces have been the key in the development of some issues of the differential geometry theory. 
We highlight the works \cite{HIMW,HIMW2,MSHS1,SX,Wang} for \textit{Translating soliton for the mean curvature flow}, when $\varphi$ is the identity and the works \cite{D,D1,RL} for \emph{singular $\alpha$-minimal surfaces}, when $\varphi (x_{3})=\alpha\text{log}(x_{3})$ for $\alpha\neq 0$. In fact,  the special case $\alpha=1$, from a physically point of view \cite{Poisson}, represents a membrane with intrinsic force vanishes under a gravitational field. 
\

Next, consider a smooth function $u:\mathcal{O}\subset\mathbb{R}^{2}\rightarrow\mathbb{R}$. We say that $\text{Graph}[u]$  is a $[\varphi,\vec{e}_{3}]$-minimal graph if and only if it solves the following elliptic equation
\begin{equation}
\label{ecuaciongrafos}
(1+u_{x}^{2})u_{yy}+(1+u_{y}^{2})u_{xx}-2u_{x}u_{u}u_{xy}=\dot{\varphi}(u)(1+u_{x}^{2}+u_{y}^{2}),
\end{equation}
where $u_{x},u_{y}$ denotes the partial derivative with respect to the first and second coordinate, respectively.  As consequence of the ellipticity of \eqref{ecuaciongrafos}, the Hopf's maximum principle hold \cite{Hopf}. The main class of $[\varphi,\vec{e}_{3}]$-minimal graphs that we will use in this paper are the $[\varphi,\vec{e}_{3}]$-catenary cylinders described in \cite{MM}. In fact, from the Theorem 3.7 of \cite{MM}, they are the only complete flat examples  together with the vertical planes and the \textit{tilted}-$[\varphi,\vec{e}_{3}]$-catenary cylinders.
\

Let $\varphi:]a,+\infty[\rightarrow ]b,c[$ with $a,b\in\mathbb{R}\cup\{-\infty\}$ and $c\in\mathbb{R}\cup\{+\infty\}$ be an strictly monotone diffeomorphism. A  $[\varphi, \vec{e}_3]$-catenary cylinder $\mathcal{G}^{h}$ is given by the cartesian product $\text{Graph}[u]\times\mathbb{R}$ where $u$  only depends of one variable, $u=u(x)$, and from \eqref{ecuaciongrafos}, it is solution of the following Cauchy's problem
\begin{equation}
 \label{htrans}
 \left\{
 \begin{array}{ll}
  u''(x)=\dot{\varphi}(u)(1+u'(x)^{2}) \\
  u(0)=h, \ \ u'(0)=0.
\end{array}
\right.
\end{equation}
The solution $u$  of \eqref{htrans}  is  even and it is defined in the  interval  $]-\Lambda_{h},\Lambda_{h}[$  with $\Lambda_{h}\in\mathbb{R}\cup\{+\infty\}$. The main properties and the asymptotic behaviour of these surfaces can be summarized in the following two results, proved in Section 3 of \cite{MM}.

\begin{theorem}\label{t2} Let $\varphi:\ ]a,+\infty[\ \longrightarrow \ ]b,c[$,  $ a,b\in \R\cup\{-\infty\}$,  $c\in  \R\cup\{\infty\}$ be a strictly increasing diffeomorphism, then the solution $u$  of \eqref{htrans} is defined in $]-\Lambda_{h},\Lambda_{h}[$, $\Lambda_{h} \in \R^{+}\cup\{+\infty\}$, it is convex,  symmetric about the $y$-axis and has a minimum at $x=0$. Moreover, 
\begin{enumerate}
\item if $c<\infty$, then $\Lambda_{h}=\infty$ and,   $ \left\{ 
\begin{array}{l}
\lim_{x\rightarrow \pm\infty} u(x)=\infty, \\
 \lim_{x\rightarrow \pm\infty} u'(x) = \pm\sqrt{\mathrm{e}^{2(c-\varphi(h))}-1}. \end{array} \right.$
 \item if $c=\infty$, 
 $ \lim_{x\rightarrow \pm\Lambda_{h}} u(x)=\infty, \quad
 \lim_{x\rightarrow \pm\Lambda_{h}} u'(x) = \pm\infty.$
\

In particular, if $\Lambda_{h}<\infty$, the graph of $u$ is asymptotic to two vertical lines. Moreover,
\item   $\Lambda_{h}<\infty$ if and only if $e^{-\varphi}\in L^{1}(]h,+\infty[)$,$\left( i.e \int_{h}^\infty e^{-\varphi(\lambda)} d\lambda < \infty \right)$ .
\item  If   $\Lambda_{\lambda}<\infty$ and $\dot{\varphi}$ is increasing (respectively, decreasing), then $\Lambda_\lambda$ is  decreasing (respectively, increasing) in $\lambda$.
\end{enumerate}

\begin{theorem}\label{t3}
 Let $\varphi:\ ]a,\infty[\ \longrightarrow \ ]b,c[$,  $ a,b\in \{\R,-\infty\}$,  $c\in \{\R,\infty\}$ be a strictly decreasing diffeomorphism, then the solution $u$  of \eqref{htrans} is defined in $]-\Lambda_{h},\Lambda_{h}[$, $\Lambda_{h} \in\R^{+}\cup\{+\infty\}$, it is concave,  symmetric about the $y$-axis and has a maximum at $x=0$. Moreover, 
\begin{enumerate}
\item if $c<\infty$, then 
$\Lambda_{h}<\infty $
  and,   $ \left\{ \begin{array}{l} \lim_{x\rightarrow \pm\Lambda_{h}} u(x)=a, \\
 \lim_{x\rightarrow \pm\Lambda_{h}} u'(x) = \pm\sqrt{\mathrm{e}^{2(c-\varphi(h))}-1}. \end{array} \right.$
\item if $c=\infty$, then 
$\Lambda_{h}<\infty $ if and only if $\int_a^{h}\mathrm{e}^{-\varphi(\lambda)} d\lambda < \infty, $
  and,   $$ \lim_{x\rightarrow \pm\Lambda_{h}} u(x)=a, \quad
 \lim_{x\rightarrow \pm\Lambda_{h}} u'(x) = \pm\infty.$$
\end{enumerate}
\end{theorem}
\begin{remark}
In the hypothesis  of Theorem \ref{t3}, the graph of $u$ is complete when $a=-\infty$. But in this case, by changing $\varphi$ by $-\varphi$, we can also apply Theorem \ref{t2}. Hereinafter, we always assume that $\varphi$ is a strictly increasing diffeomorphism.
\end{remark}

\begin{center}
\includegraphics[scale=0.17]{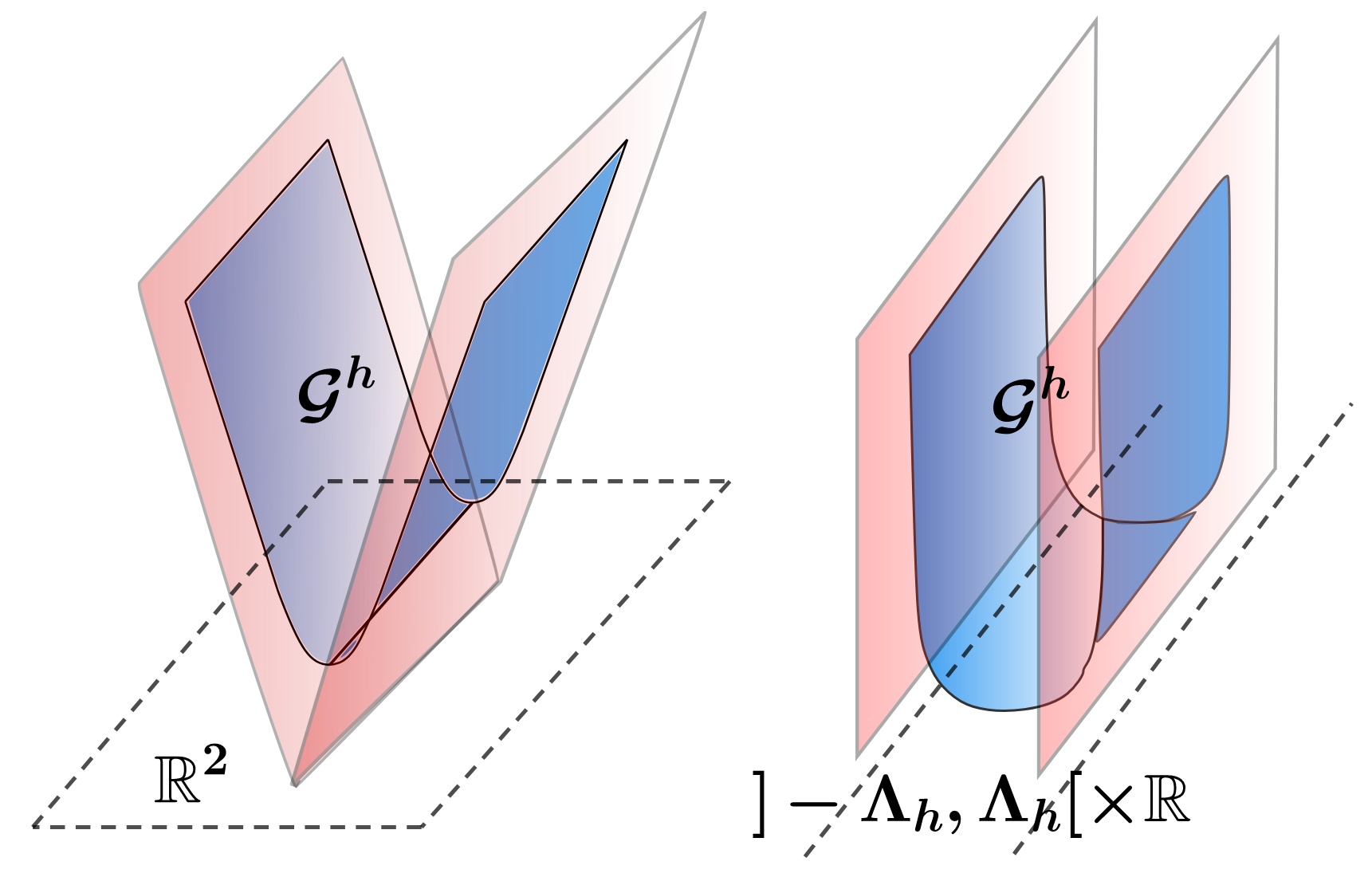}
\end{center}
\end{theorem}
The goal of this paper is a uniqueness result for the $[\varphi,\vec{e}_{3}]$-catenary cylinders by their asymptotic behaviour. Roughly, if we have a surface $\Sigma$ that looks like a $[\varphi,\vec{e}_{3}]$-catenary cylinder $\mathcal{G}^{h}$ at infinity, in the sense of the definition \ref{convergenciaG} in Section 2, then our surface coincides with some $[\varphi,\vec{e}_{3}]$-catenary cylinder with the same behaviour $\mathcal{G}^{h}$. Motivated by the work \cite{MSHS1} of F. Mart\'in, J. P\'erez-Garc\'ia, A. Savas-Halilaj and K. Smoczyk  for the grim reaper cylinder translating soliton, we increase the family of $[\varphi,\vec{e}_{3}]$-minimal surfaces where these types of results hold under different assumptions of asymptotic behaviour when $\varphi$ is a strictly increasing convex diffeomorphism such that $e^{-\varphi}$ is integrable. From \cite{MM,MMJ}, this family of functions $\varphi$ is the natural candidate to consider for generalizing the result of \cite{MSHS1}. The statement of the main result of this paper tells us the following 

\begin{theorem}
\label{uni}
Let $\varphi:]a,+\infty[\rightarrow ]b,c[$, $a,b\in\mathbb{R}\cup\{-\infty\}$ and $c\in\mathbb{R}\cup\{+\infty\}$ be a strictly increasing convex diffeomorphism such that $e^{-\varphi}\in L^{1}(]a,+\infty[)$ and bounded quotient $\ddot{\varphi}/\dot{\varphi}$. If $\Sigma$ is a complete connected $[\varphi,\vec{e}_{3}]$-minimal graph $\mathcal{C}^{\infty}$-asymptotic to $[\varphi,\vec{e}_{3}]$-catenary cylinder $\mathcal{G}^{h}$, outside a cylinder, for some $h\in ]a,+\infty[$, then $\Sigma$ coincides with some $[\varphi,\vec{e}_{3}]$-catenary cylinder with the same behaviour that $\mathcal{G}^{h}$.
\end{theorem}

We would like to point out that we generalize this uniqueness for $[\varphi,\vec{e}_{3}]$-minimal surface assuming that $\Sigma$ is a vertical graph. The main advantage of this hypothesis is that, we only need to prove that the angle function $\langle N,\vec{e}_{2}\rangle$ vanishes everywhere because in such case,  $\Sigma$ is invariant by translations in the direction $\vec{e}_{2}$ and so, from \cite{MM}, it must coincide with some $[\varphi,\vec{e}_{3}]$-catenary cylinder. Otherwise, this conclusion is not true in general. The uniqueness of the grim reaper translating soliton is guaranteed by  \cite{MSHS1}, since it was showed that the Gauss curvature vanishes everywhere and applying the Theorem B of \cite{MSHS2}, $\Sigma$ must be a grim reaper cylinder. 
\

The paper is organized as follows. In Section 2, we introduce the notation together with some fundamental equations  that we will use throughout the work.  In Section 3, we prove a compactness result to studying the infinity of our surfaces under horizontal translations. Finally, the proof of Theorem \ref{uni} appears in Section 4. We will show that our graphs are invariant in the direction $\vec{e}_{2}$ because the function $\langle N,\vec{e}_{2}\rangle$ vanishes everywhere due to the asymptotic behaviour of $\Sigma$.

\section{Preliminaries}

In this section, we introduce the notations, definitions and the fundamental equations that we will use throughout the paper.

For an orientable immersion $\Sigma$ in $\mathbb{R}^{3}$, we denote by $\nabla$, $\Delta$ the gradient and the Laplacian on $\Sigma$, respectively. The second fundamental form on $\Sigma$ will be denoted by $\mathcal{S}$, so that $H=\text{trace}(\mathcal{S})$, $K$, and $\vert\mathcal{S}\vert^{2}$ stand for the mean curvature, the Gauss curvature and the squared norm of the second fundamental form on $\Sigma$, respectively. Besides that, let us define the coordinate functions $x_{i}:\Sigma\rightarrow\mathbb{R}$ by $x_{i}(p)=\langle p,\vec{e}_{i}\rangle$ and the angle functions $\eta_{i}:\Sigma\rightarrow\mathbb{R}$ by $\eta_{i}(p)=\langle\vec{e}_{i},N(p)\rangle$, for any $p\in\Sigma$.

Bearing in mind our goals, we will need a good control of our surfaces in the infty motivating the following definitions.
\begin{definition}
Consider a family $t$-parameter of vertical planes $$\Pi (t)=\{(x_{1}, x_{2}, x_{3})\in\mathbb{R}^{3}: x_{1}=t\}$$ and $\Sigma$ be a subset of $\mathbb{R}^{3}$. The right part of $\Sigma$ with respect to the plane $\Pi(t)$ is defined by the subset $$\Sigma_{+}(t)=\{ (x_{1}, x_{2}, x_{3})\in\Sigma: x_{1}>t\},$$ and the left part with respect to $\Pi(t)$ is defined by $\Sigma_{-}(t)=\Sigma \backslash\overline{\Sigma_{+}(t)}.$ The reflection of $\Sigma_{+}(t)$ with respect to the plane $\Pi(t)$ is given by $$\Sigma_{+}^{*}(t)=\{(2t-x_{1},x_{2},x_{3})\in\mathbb{R}^{3}: (x_{1},x_{2},x_{3})\in\Sigma_{+}(t)\}.$$
\end{definition}

Taking into account the shape of the $[\varphi,\vec{e}_{3}]$-catenary cylinders given by the Theorem \ref{t2}. The following definition is the key to understand the smooth convergence of a surface to $[\varphi,\vec{e}_{3}]$-catenary cylinder.

\begin{definition}
\label{convergenciaG}
Let $\varphi: ]a,+\infty[\rightarrow]b,c[$ $a,b\in \mathbb{R}\cup\{-\infty\}$ , $c\in\mathbb{R}\cup\{+\infty\}$ be a strictly increasing diffeomorphism such that $e^{-\varphi} \in L^{1}(]a,+\infty[)$ and $\mathcal{G}^{h}$ be a $[\varphi,\vec{e}_{3}]$-catenary cylinder in $\mathbb{R}^{3}_{a}=\{p\in\mathbb{R}^{3}: \langle p,\vec{e}_{3}\rangle\geq a\}$ for some $h\in ]a,+\infty[$. 
\\

We will say that a smooth surface $\Sigma$ is $\mathcal{C}^{k}$-asymptotic to the \textit{right part} $\mathcal{G}^{h}_{+}(0)$ of $\mathcal{G}^{h}$  if for any $\varepsilon>0$ there exists $\delta>0$ such that $\Sigma$  can parametrized as a graph over $\mathcal{G}^{h}$ as follows $$\widetilde{F}:T_{\delta,h}^{+}\subset\mathbb{R}^{2}\rightarrow\mathbb{R}^{3} \ \  \ \ \widetilde{F}=F+\overline{u}\, N_{F},$$
where $T_{\delta,h}^{+}:=]\Lambda_{h}-\delta,\Lambda_{h}[\times\mathbb{R}$, $F(x_{1},x_{2})=(x_{1},x_{2},u(x_{1}))$ parametrizes $\mathcal{G}^{h}$ on $T_{\delta,h}^{+}$, $u$ is a solution of \eqref{htrans} with $u(0)=h$, $\overline{u}:T_{\delta,h}^{+}\rightarrow\mathbb{R}$ is a function in $\mathcal{C}^{k}(T_{\delta,h}(+))$ such that
$$\sup_{T_{\delta,h}^{+}}\vert\overline{u}\vert <\varepsilon \ \ , \ \ \sup_{T_{\delta,h}^{+}}\vert D^{j}\overline{u}\vert<\varepsilon, \text{ for any } j\in\{1,\cdots,k\}.$$
and $N_{F}$ is the downwards unit normal of $\mathcal{G}^{h}$. Analogously,  we will say that a smooth surface $\Sigma$ is $\mathcal{C}^{k}$-asymptotic to left \textit{left part} $\mathcal{G}^{h}_{-}(0)$ of $\mathcal{G}^{h}$  if for any $\varepsilon>0$ there exists $\delta>0$ such that $\Sigma$  can parametrized as a graph over $\mathcal{G}^{h}$ as follows $$\widetilde{F}:T_{\delta,h}^{-}\subset\mathbb{R}^{2}\rightarrow\mathbb{R}^{3} \ \  \ \ \widetilde{F}=F+\overline{u}\, N_{F},$$
where $T_{\delta,h}^{-}:=]-\Lambda_{h},-\Lambda_{h}+\delta[\times\mathbb{R}$, $F(x_{1},x_{2})=(x_{1},x_{2},u(x_{1}))$ parametrizes $\mathcal{G}^{h}$ on $T_{\delta,h}^{-}$, $u$ is a solution of \eqref{htrans} with $u(0)=h$, $\overline{u}:T_{\delta,h}^{-}\rightarrow\mathbb{R}$ is a function in $\mathcal{C}^{k}(T_{\delta,h}^-)$ such that
$$\sup_{T_{\delta,h}^{-}}\vert\overline{u}\vert <\varepsilon \ \ , \ \ \sup_{T_{\delta,h}^{-}}\vert D^{j}\overline{u}\vert<\varepsilon, \text{ for any } j\in\{1,\cdots,k\}.$$

In particular, we say that $\Sigma$ is $\mathcal{C}^{k}$-asymptotic to $\mathcal{G}^{h}$ if and only if $\Sigma$ is $\mathcal{C}^{k}$-asymptotic to the both branches $\mathcal{G}^{h}_{+}(0)$ and $\mathcal{G}^{h}_{-}(0)$. Moreover, a smooth surface $\Sigma$ is called $\mathcal{C}^{k}$-asymptotic to $\mathcal{G}^{h}$, outside a cylinder, if  there exists a solid cylinder $\mathfrak{c}$ whose axis is $\mathcal{G}^{h}\cap\Pi(0)$ and the set $\Sigma-\mathfrak{c}$ consists of two connected components $\Sigma_{1}$ and $\Sigma_{2}$ which are $\mathcal{C}^{k}$-asymptotic to $\mathcal {G}^{h}_{+}(0)$ and $\mathcal{G}^{h}_{-}(0)$, respectively.
\end{definition}

 Next, we compute the fundamental equations that we will use in Section 4.  Following the proof of Lemma 2.1 in \cite{MM}, we get the following result.
\begin{lemma}
The following equations holds
\begin{align}
&\Delta x_{i}+\langle\nabla\varphi,\nabla x_{i}\rangle=0 \ \ \text{ for } i=1,2 \label{laplacianox1} \\
&\Delta x_{3}+\langle\nabla\varphi,\nabla x_{3}\rangle=\dot{\varphi} , \label{laplacianoaltura} \\
&\Delta \eta_{i}+\langle\nabla\varphi,\nabla\eta_{i}\rangle+\vert\mathcal{S}\vert^{2}\eta_{i}=\ddot{\varphi}\eta_{i}\eta_{3}^{2}  \ \ \text{for } i=1,2  \label{laplacianoeta1} \\
&\Delta\eta_{3}+\langle\nabla\varphi,\nabla\eta_{3}\rangle+\vert\mathcal{S}\vert^{2}\eta_{3}=-\ddot{\varphi}\eta_{3}\vert\nabla x_{3}\vert^{2}. \label{laplacianoeta}
\end{align}
\end{lemma}

\begin{proof}
Fix  $p\in\Sigma$ and consider $\{v_{j}\}_{j=1,2}$ an orthonormal frame of $T_{p}\Sigma$. Then,
\begin{equation}
\label{lapla}
\Delta x_{i}=\sum_{j=1}^{2}\langle\nabla_{v_{j}}\nabla x_{i},v_{j}\rangle,
\end{equation}
where $\nabla_{v_{j}}$ is the Levi-Civita connection on $\Sigma$. Moreover, it is clear that
\begin{equation}
\label{grad}
\nabla x_{i}=\vec{e}_{i}^{T}=\vec{e}_{i}-\langle \vec{e}_{i}, N\rangle N \ \ \text{and} \ \ \nabla\eta_{i}=\nabla_{\vec{e}_{i}^{T}}N.
\end{equation}
From \eqref{lapla}, \eqref{grad} and the Laplace-Beltrami equation, we have that
\begin{equation}
\label{lapla2}
\Delta x_{i}=-\langle \vec{e}_{i},N\rangle\sum_{j=1}^{2}\langle \nabla_{v_{j}} N,v_{j}\rangle=-\langle \vec{e}_{i}, N\rangle H.
\end{equation}
Hence, from the definition \eqref{defphimin} and the equation \eqref{lapla2}, it is proved that
\begin{align*}
\Delta x_{i}=&\dot{\varphi}\langle \vec{e}_{i}, N\rangle\langle \vec{e}_{3}, N\rangle=\dot{\varphi}\langle \vec{e}_{3},\vec{e}_{i}-\vec{e}_{i}^{T}\rangle=-\langle\nabla\varphi,\nabla x_{i}\rangle \ \ i=1,2. \\
&\Delta x_{3}=\dot{\varphi}\langle \vec{e}_{3}, N\rangle^{2}=\dot{\varphi}(1-\vert\nabla x_{3}\vert^{2})=\dot{\varphi}-\langle\nabla\varphi,\nabla x_{3}\rangle.
\end{align*}
Next, from \eqref{defphimin} and the well known equation $\Delta N=\nabla H-\vert\mathcal{S}\vert^{2} N$
\begin{equation*}
\label{laplaN}
\Delta N+\dot{\varphi}\nabla\eta_{3}+\ddot{\varphi}\eta_{3}\nabla x_{3}+\vert\mathcal{S}\vert^{2}N=0.
\end{equation*}
Consequently,
\begin{align}
\Delta\eta_{i}+&\vert\mathcal{S}\vert^{2}\eta_{i}=-\dot{\varphi}\langle\nabla\eta_{3}, \vec{e}_{i}^{T}\rangle-\ddot{\varphi}\eta_{3}\langle\nabla x_{3},\vec{e}_{i}^{T}\rangle \text{ for } i=1,2  \label{laplaeta1} \\
&\Delta\eta_{3}+\vert\mathcal{S}\vert^{2}\eta_{3}=-\dot{\varphi}\langle\nabla\eta_{3},\vec{e}_{3}^{T}\rangle-\ddot{\varphi}\eta_{3}\vert\nabla x_{3}\vert^{2}.  \label{laplaeta3}
\end{align}
On the other hand, by a simple computation
\begin{align}
\dot{\varphi}\langle\nabla\eta_{3}, \vec{e}_{i}^{T}\rangle=\dot{\varphi}\mathcal{S}(\nabla x_{3}, \vec{e}_{i}^{T})=\langle\nabla\varphi,\nabla \eta_{i}\rangle \text{ for } i=1,2,3. \label{gradeta1}
\end{align}
Then, the lemma follows from \eqref{laplaeta1}, \eqref{laplaeta3} and \eqref{gradeta1}.
\end{proof}
From the previous expressions \eqref{laplacianox1},\eqref{laplacianoaltura},\eqref{laplacianoeta1} and \eqref{laplacianoeta}, we can prove the following result using the same equations as in \cite{MMJ}.
\begin{proposition}
Let $\varphi:]a,+\infty[\rightarrow]b,c[$ , $a,b\in\mathbb{R}\cup\{-\infty\}$ and $c\in\mathbb{R}\cup\{+\infty\}$ be a smooth function and $\Sigma$ be a strictly mean convex $[\varphi,\vec{e}_{3}]$-minimal immersion in $\mathbb{R}^{3}_{a}$. Then,
\begin{align}
\Delta^{\varphi}\left(\frac{\eta_{2}}{\eta_{3}}\right)+2\langle\nabla\left(\frac{\eta_{2}}{\eta_{3}}\right),\frac{\nabla\eta_{3}}{\eta_{3}}\rangle=\ddot{\varphi}\frac{\eta_{2}}{\eta_{3}} \label{LAP2}.
\end{align}
where $\Delta^{\varphi}=\Delta+\langle\nabla\varphi,\nabla\cdot\rangle$ is the drift Laplacian of $\Sigma$.
\end{proposition}

\begin{proof}
By a straightforward computation, for any smooth functions $f,g:\Sigma\rightarrow\mathbb{R}$ with $g\neq 0$ everywhere we have that,
\begin{equation}
\label{expre}
\Delta^{\varphi}\left(\frac{f}{g}\right)+2\langle\nabla\left(\frac{f}{g}\right),\frac{\nabla g}{g}\rangle=\frac{g\, \Delta^{\varphi}f-f\, \Delta^{\varphi} g}{g^{2}}.
\end{equation}
Consequently, the proof follows applying \eqref{expre} in the equation \eqref{LAP2} together with \eqref{laplacianox1}, \eqref{laplacianoaltura}, \eqref{laplacianoeta1} and \eqref{laplacianoeta}.
\end{proof}
From the ellipticity of the equation \eqref{LAP2} and bearing in mind the goal of proving that the function $\eta_{2}$ vanishes everwyhere, we finish this section writing the previous quotient $\eta_{2}/\eta_{3}$ as a function over the $[\varphi,\vec{e}_{3}]$-catenary cylinder to getting the control of this function at infinity. In Section 4, we will see that the quotient $\eta_{2}/\eta_{3}$ tends to zero and, from a strong maximum principle, we will prove that $\eta_{2}/\eta_{3}$ vanishes everywhere. In particular, the angle $\eta_{2}$ will always be zero as we wanted.

Consider a $[\varphi,\vec{e}_{3}]$-catenary cylinder $\mathcal{G}^{h}$ contained in the slab $]-\Lambda_{h},\Lambda_{h}[\times\mathbb{R}^{2}$ (posibly all $\mathbb{R}^{3}$)  parametrized by
$$F(x_{1},x_{2})=(x_{1},x_{2},u(x_{1})) \ \ (x_{1},x_{2})\in ]-\Lambda_{h},\Lambda_{h}[\times\mathbb{R},$$
where $u$ is a smooth function satisfying \eqref{htrans} with $u(0)=h$ for some $h\in ]a,+\infty[$. Now, suppose that our immersion $\Sigma$ is a vertical $[\varphi,\vec{e}_{3}]$-minimal graph $\mathcal{C}^{\infty}$-asymptotic to $\mathcal{G}^{h}$, outside a cylinder. Then, for any $\varepsilon>0$ there exists $\delta>0$ such that the right part $\Sigma_{+}(\Lambda_{h}-\delta)$  can be parametrized by
$\widetilde{F}:T_{\delta,h}^{+}\rightarrow\mathbb{R}^{3}$ where $ \widetilde{F}=F+\overline{u}N_{F}$ and $\overline{u}:T_{\delta,h}^{+}\rightarrow\mathbb{R}$ is a smooth function such that
\begin{equation*}
\sup_{T_{\delta,h}^{+}}\{\vert\overline{u}\vert\}<\varepsilon \ \ \text{and} \ \ \sup_{T_{\delta,h}^{+}}\{\vert D^{j}\overline{u}\vert\}<\varepsilon \ \ \text{for any}\ \ j\in\mathbb{N},
\end{equation*}
and $N_{F}$ is the unit normal vector of $\mathcal{G}^{h}$ given by,
\begin{equation}
 \label{normalF}
N_{F}=\left( \frac{u'}{\sqrt{1+u'^{2}}},0,-\frac{1}{\sqrt{1+u'^{2}}}\right).
\end{equation}
Next, we compute the Gauss map $N_{\widetilde{F}}$ of $\Sigma$ with respect to the following orthogonal frame $\{\frac{\partial F}{\partial x_{1}},\frac{\partial F}{\partial x_{2}},N_{F}\}.$
Notice that,
\begin{equation}
\label{e1}
\frac{\partial F}{\partial x_{1}}=(1,0,u') \ \ \text{and} \ \ \frac{\partial F}{\partial x_{2}}=(0,1,0).
\end{equation}
If we denote by
\begin{equation}
\label{e2}
E_{1}=\left( \frac{1}{\sqrt{1+u'^{2}}},0,\frac{u'}{\sqrt{1+u'^{2}}}\right) \ \ \text{and} \ \ E_{2}=\frac{\partial F}{\partial x_{2}},\end{equation}
then ,
\begin{equation}
\label{e3}
\frac{\partial N_{F}}{\partial x_{1}}=\dot{\varphi}E_{1} \ \ \text{and} \ \ \frac{\partial N_{F}}{\partial x_{2}}=0.
\end{equation}
Consequently, from \eqref{e1}, \eqref{e2} and \eqref{e3}, we get that
\begin{align}
&\frac{\partial\widetilde{F}}{\partial x_{1}}=(\sqrt{1+u'^{2}}+\overline{u}\dot{\varphi})E_{1}+\overline{u}_{x_{1}}N_{F}, \label{v1} \\
&\frac{\partial\widetilde{F}}{\partial x_{2}}=E_{2}+\overline{u}_{x_{2}}N_{F}, \label{v2} \\
\frac{\partial\widetilde{F}}{\partial x_{1}}\wedge\frac{\partial\widetilde{F}}{\partial x_{2}}&=-\overline{u}_{x_{1}}E_{1}-(\sqrt{1+u'^{2}}+\overline{u}\dot{\varphi})\overline{u}_{x_{2}}E_{2}+(\sqrt{1+u'^{2}}+\overline{u}\dot{\varphi})N_{F} \label{v3}.
\end{align}
Thus, from \eqref{v1}, \eqref{v2} and \eqref{v3}, the normal $N_{\widetilde{F}}$ can be written as
\begin{equation}
\label{NormalFF}
N_{\widetilde{F}}=\frac{\left(-\frac{\overline{u}_{x_{1}}}{1+u'^{2}}\right)\frac{\partial F}{\partial x_{1}}-\overline{u}_{x_{2}}\left(1+\overline{u}\frac{\dot{\varphi}}{\sqrt{1+u'^{2}}}\right)\frac{\partial F}{\partial x_{2}}+\left(1+\overline{u}\frac{\dot{\varphi}}{\sqrt{1+u'^{2}}}\right) N_{F}}{\sqrt{\frac{\overline{u}_{x_{1}}^{2}}{1+u'^{2}}+(1+\overline{u}_{x_{2}}^{2})\left(1+\overline{u}\frac{\dot{\varphi}}{\sqrt{1+u'^{2}}}\right)^{2}}}.
\end{equation}
Finally, from \eqref{NormalFF}, we compute the previous quotient $\eta_{2}/\eta_{3}$  over $\mathcal{G}^{h}$ as we wanted. From the equations \eqref{normalF} and \eqref{e1}, we get that
\begin{equation}
\label{secoorde}
\langle\frac{\partial F}{\partial x_{1}},\vec{e}_{2}\rangle=0,  \ \ \langle\frac{\partial F}{\partial x_{2}},\vec{e}_{2}\rangle=1, \ \ \langle N_{F},\vec{e}_{2}\rangle=0,
\end{equation}
\begin{equation}
\label{thirdcoorde}
\langle\frac{\partial F}{\partial x_{1}},\vec{e}_{3}\rangle=u',  \ \ \langle\frac{\partial F}{\partial x_{2}},\vec{e}_{3}\rangle=0, \ \ \langle N_{F},\vec{e}_{3}\rangle=-\frac{1}{\sqrt{1+u'^{2}}}.
\end{equation}
Consequently, from \eqref{NormalFF}, \eqref{secoorde} and \eqref{thirdcoorde}, we prove the following result.
\begin{proposition}Under the assumptions above, for any $\varepsilon>0$, there exists $\delta>0$ small enough such that the following formula holds
\label{formula}
\begin{equation}
\label{HFF}
\frac{\eta_{2}}{\eta_{3}}=\overline{u}_{x_{2}}\sqrt{1+u'^{2}}\left(\frac{1+\overline{u}\left(\frac{\dot{\varphi}}{\sqrt{1+u'^{2}}}\right)}{1+\overline{u}_{x_{1}}\frac{u'}{\sqrt{1+u'^{2}}}+\overline{u}\left(\frac{\dot{\varphi}}{\sqrt{1+u'^{2}}}\right)}\right) \text{ on } T_{\delta,h}^{\pm} 
\end{equation}
respectively, where $\overline{u}:T_{\delta,h}^{\pm}\times\mathbb{R}\rightarrow\mathbb{R}$ is a smooth function defined in \ref{convergenciaG} such that
\begin{equation*}
\sup_{T_{\delta,h}^{\pm}}\{\vert\overline{u}\vert\}<\varepsilon \ \ \text{and} \ \ \sup_{T_{\delta,h}^{\pm}}\{\vert D^{j}\overline{u}\vert\}<\varepsilon \ \ \text{for any}\ \ j\in\mathbb{N}.
\end{equation*}
\end{proposition}

\begin{remark}
\label{curvaturetendszero}
From Theorem \eqref{t2} and  equations \eqref{normalF}-\eqref{e3}, if $\varphi$ is a smooth function such that $e^{-\varphi}$ is not integrable and the $j$-th derivative of $\varphi$ is bounded  for any $j\in\{1,\cdots, k\}$, then $\Sigma$ is closed to $\mathcal{G}^{h}$ in $C^{k}$-topology.
\end{remark}

\section{Compactness Theorem from a Barrier Maximum Principle}
 A fundamental step for the uniqueness of the $[\varphi,\vec{e}_{3}]$-catenary cylinder is his behaviour under translations in the direction $\vec{e}_{2}$. In section 4, we will prove that the translations in the direction $\vec{e}_{2}$ of any $[\varphi,\vec{e}_{3}]$-minimal graph $\Sigma$  converge to some $[\varphi,\vec{e}_{3}]$-catenary cylinder if $\Sigma$ is $\mathcal{C}^{\infty}$-asymptotic to some $\mathcal{G}^{h}$, outside a cylinder. Our first result of this section is a direct application of a compactness theorem of B. White for minimal surfaces in 3-manifolds (Theorem 1.1 of \cite{W2}). Next, we state a Barrier Maximun principle for the blow-up points of the area proved in \cite{W3}. Finally, as a consequence of these results together with the existence of $[\varphi,\vec{e}_{3}]$-Bowls given in \cite{MM}, we prove a compactness result for surfaces $\mathcal{C}^{\infty}$-asymptotic  to $\mathcal{G}^{h}$, outside either a cylinder or a compact set.
\begin{theorem}
\label{CW}
Let $\Omega$ be an open subset of $\mathbb{R}^{3}$. Let $\{\varphi_{n}\}_{n\in\mathbb{N}}$ be a sequence of smooth functions on $\Omega$ converging smoothly to a $\varphi_{\infty}$.  Let $\{\Sigma_{n}\}_{n\in\mathbb{N}}$ be a sequence of properly embedded minimal surface in $\Omega^{\varphi_{n}}$. Suppose also that the area and the genus of the sequence $\{\Sigma_{n}\}$ are uniformly bounded on compacts subsets of $\Omega$. Then, after passing to a subsequence, $\{\Sigma_{n}\}$ converges to a smooth properly embedded minimal $\Sigma_{\infty}$ in $\Omega^{\varphi_{\infty}}$ and the convergence is smooth away from a discrete set $\Gamma$.
 For each connected component $\mathcal{S}$ of $\Sigma_{\infty}$ either,
\begin{enumerate}
\item the convergence to $\mathcal{S}$ is smooth everywhere with multiplicity $1$, or
\item the convergence to $\mathcal{S}$ is smooth with some multiplicity $>1$ away from $\mathcal{S}\cap\Gamma$. In this case, if $\mathcal{S}$ is two-sided, then it must be stable.
\end{enumerate}
If the total curvatures of $\Sigma_{n}$ are bounded by $\beta$, the set $\Gamma$ has at most $\beta/4\pi$ points.
\end{theorem}

A crucial assumption in the Theorem of B. White is that the sequence $\{\Sigma_{n}\}$ has uniformly bounded area on compact subset of $\Omega$. Next, we discuss when this hypothesis hold and prove a result of compactness.
\begin{remark}
\label{acoarea}
 In \cite{MMJ}, it is proved that if $\varphi$ is strictly increasing convex smooth function with $\sup\{\ddot{\varphi}-\dot{\varphi}^{2}\}<+\infty$, then any sequence  $\{\Sigma_{n}\}$ of mean convex $[\varphi,\vec{e}_{3}]$-minimal immersion has uniformly bounded intrinsic area on compact subsets of $\Omega$.
\end{remark}

Let us denote by,
\begin{equation*}
\label{blowupset}
\mathcal{Z}:=\{p\in\Omega : \limsup_{n\rightarrow+\infty}\mathcal{A}^{\varphi}(\Sigma_{n}\cap B_{r}(p))=+\infty \text{ for any } r>0\},
\end{equation*}
where $B_{r}(p)$ is an Euclidean ball centered in $p$ of radius $r$. If we show that $\mathcal{Z}=\varnothing$, then $\{\Sigma_{n}\}$ has locally uniformly bounded area. From the Theorems 2.6 and 7.3 in \cite{W3}, White proved that $\mathcal{Z}$ satisfies the following maximum principle as properly embedded minimal surfaces without boundary. 

\begin{theorem}{\rm (Strong Barrier Principle)}
\label{barrierprinciple}
Let $\Omega$ be an open subset of $\mathbb{R}^{3}$.  Let $\{\varphi_{n}\}_{n\in\mathbb{N}}$ be a sequence of smooth function converging smoothly to a smooth functions $\varphi_{\infty}$ and $\{\Sigma_{n}\}_{n\in\mathbb{N}}$ be a sequence of minimal surfaces in $\Omega^{\varphi_{n}}$ such that  the length of $\partial\Sigma_{n}$ are uniformly bounded on compact subsets of $\Omega$ and $D$ be a closed region of $\Omega$ with smooth connected boundary such that $\mathcal{Z}\subset D$ and $\langle H_{\partial D},\nu\rangle\leq 0$ where $\nu$ is the unit normal of $\partial D$ points into $D$ and $H_{\partial D}$ is the mean curvature of the boundary of $D$. If $\mathcal{Z}$ contains any point of $\partial D$, then $\partial D\subset\mathcal{Z}$.
\end{theorem}
As a consequence of the result above, we give the following compactness theorem showing that the set $\mathcal{Z}=\varnothing$.
\begin{theorem}
\label{conv}
Let  $\varphi:]a,+\infty[\rightarrow ]b,c[$ $a,b\in\mathbb{R}\cup\{-\infty\}$ and $c\in\mathbb{R}\cup\{+\infty\}$ be a strictly increasing convex diffeomorphism and $\Sigma$ be a connected properly embedded $[\varphi,\vec{e}_{3}]$-minimal minimal surface with locally bounded genus $\mathcal{C}^{\infty}$-asymptotic to $\mathcal{G}^{h}$, outside a cylinder, for some $h\in ]a,+\infty[$. Suppose that $\{b_{n}\}_{n\in\mathbb{N}}$ is a divergent sequence of real numbers and consider the sequence of $[\varphi,\vec{e}_{3}]$-minimal surfaces $\{\Sigma_{n}=\Sigma-(0,b_{n},0)\}_{n\in\mathbb{N}}$ . Then, after passing to a subsequence, $\Sigma_{n}$ converge smoothly with multiplicity one to a properly embedded connected $[\varphi,\vec{e}_{3}]$-minimal $\Sigma_{\infty}$ which has the same asymptotic behaviour that $\Sigma$.
\end{theorem}

\begin{proof}
By assumption our surface $\Sigma$ is $\mathcal{C}^{\infty}$-asymptotic to $\mathcal{G}^{h}$,  outside a cylinder $\mathfrak{c}$. Up to rigid motions, we can consider $\mathfrak{c}$ such that $$\mathfrak{c}=\{(x_{1},x_{2},x_{3})\in\mathbb{R}^{3}: x_{1}^{2}+x_{3}^{2}\leq r_{0}^{2}\}.$$
From Definition \ref{convergenciaG}, each $\Sigma_{n}\backslash\mathfrak{c}$ consists in two connected components $\Sigma_{n,1},\Sigma_{n,2}$ $\mathcal{C}^{\infty}$-asymptotic to $\mathcal{G}^{h}_{+}(0)$ and $\mathcal{G}^{h}_{-}(0)$, respectively. For an arbitrary $z\in\mathbb{R}$ denote by $$\mathcal{R}_{z}^{+}=\{p\in\mathbb{R}^{3}: \langle p,\vec{e}_{3}\rangle\geq z\} \text{ and } \mathcal{R}^{-}_{z}=\mathbb{R}^{3}\backslash\mathcal{R}^{+}_{z}.$$
Split each surface $\Sigma_{n}$ of the surface into the parts
$$\Sigma_{n,k}^{+}(z)=\Sigma_{n,k}\cap\mathcal{R}_{z}^{+} \ \ \text{with} \ \ k\in\{1,2\}\ \ \text{and} \ \ \Sigma_{n}^{-}(z)=\Sigma_{n}\backslash\left(\bigcup_{k=1}^{2}\Sigma_{n,k}^{+}(z)\right).$$

\textbf{Claim 1.} There exists $z_{1}$ large enough such that for any $z\geq z_{1}$ the sequence $\{\Sigma_{n,k}^{+}(z)\}_{n\in\mathbb{N}}$ have uniformly bounded area on compact subsets of $\Omega$ with respect the Ilmanen's metric.
\begin{proof}[Proof of the Claim 1.]
We only argue on the case $k=1$ since the case $k=2$ is analogous. Let $\mathcal{K}$ be a compact subset of $\mathbb{R}^{3}$ and $B(0,r)$ the ball of radius $r$ centered at the origin containing $\mathcal{K}$. Denote by $\mathcal{U}_{n,z}$ the projection of $\Sigma_{n,1}^{+}(z)\cap\mathcal{K}$ to $]-\Lambda_{h},\Lambda_{h}[\times\mathbb{R}$. From the definition \ref{convergenciaG} and taking $z_{1}$ large enough, we know that for any $z\geq z_{1}$ and any $\varepsilon>0$ there exists $\delta$ (depending only of $\varepsilon$) such that we can parametrize $\Sigma_{n,1}^{+}(z)\cap\mathcal{K}$ by $\widetilde{F}_{n}$ over $\mathcal{U}_{n,z}\cap T_{\delta,h}^{+}$. Hence, the area on this compact subset is given by,
\begin {equation*}
\label{areaK}
\mathcal{A}^{\varphi}(\Sigma_{n,1}^{+}(z)\cap\mathcal{K})=\int_{\mathcal{U}_{n,z}\cap T_{\delta,h}^{+}}\, e^{\varphi(\langle\widetilde{F}_{n}(x_{1},x_{2}),\vec{e}_{3}\rangle)}\bigg\vert\frac{\partial\widetilde{F}_{n}}{\partial x_{1}}\wedge \frac{\partial\widetilde{F}_{n}}{\partial x_{2}}\bigg\vert\, dx_{1}\, dx_{2}.
\end{equation*}
Consequently, from the monoticity of $\varphi$, the Theorem \ref{t2} and the expression \eqref{v3}, there must be a constant  $C(R,\varepsilon)>0$ (depending only of $R$ and $\varepsilon$) such that
\begin{equation*}
\label{cot1}
\mathcal{A}^{\varphi}(\Sigma_{n,1}^{+}(z)\cap\mathcal{K})\leq e^{\varphi(R+\varepsilon)}C(R,\varepsilon)\mathcal{A}(\mathcal{U}_{n,z})\leq e^{\varphi(R+\varepsilon)}C(R,\varepsilon)\mathcal{A}(\mathcal{K}).
\end{equation*}

\end{proof}
\noindent\textbf{Claim 2.} There exists $z_{2}\geq z_{1}$ such that for any $z\geq z_{2}$ the sequence of surfaces $\{\Sigma_{n}^{-}(z)\}_{n\in\mathbb{N}}$ have uniformly bounded area on compact sets of $\Omega$ with respect to the Ilmanen's metric.
\begin{proof}[Proof of the Claim 2.]
Proceeding as before, the sequence $\{\partial\Sigma_{n}^{-}(z)\}$ has uniformly bounded length on compact sets. Notice that each $\partial\Sigma_{n}^{-}(z)$ has two connected components $\partial\Sigma_{n,k}^{-}(z)$ with $k\in\{1,2\}$. Again, we will only argue on $k=1$ since the same reasoning works for $k=2$. Fix a compact subset $\mathcal{K}$ of $\mathbb{R}^{3}$ and $B(0,r)$ the ball of radius $r$ centered in the origin containing $\mathcal{K}$. Taking $z_{2}\geq z_{1}$ large enough, we get that for any $z\geq z_{2}$ and any $\varepsilon>0$ there exists $\delta$ (depending only of $\varepsilon$) such that $\partial\Sigma_{n,1}^{-}(z)\cap\mathcal{K}$ can be represented as a planar curve $\gamma_{n}:I_{n}\subset\mathbb{R}\rightarrow\Pi(r_{n})$ given by $\gamma_{n}(t)=\widetilde{F}_{n}(r_{n},t)$ with $r_{n}\in]\Lambda_{h}-\delta,\Lambda_{h}[$, $I_{n}$ stands the projection of $\partial\Sigma_{n,1}^{-}(z)\cap\mathcal{K}$ to $\{r_{n}\}\times\mathbb{R}$ and $\widetilde{F}_{n}$ defined in \ref{convergenciaG}. Hence, from the equation \eqref{v2}, we can estimate the length as follows
$$\int_{I_{n}}\, e^{\varphi(\langle\gamma_{n}(t),\vec{e}_{3}\rangle)}\Vert\gamma_{n} '(t)\Vert\, dt \leq e^{\varphi (R+\varepsilon)}\int_{I_{n}}\Vert\gamma_{n} '(t)\Vert\, dt\leq  C(\varepsilon,r)e^{\varphi (R+\varepsilon)},$$
where $C(\varepsilon,r)$ is a positive constant that only depends of $\varepsilon$ and $r$.

On the other hand, from the Claim 1, we get that $\mathcal{Z}$ is contained within a horizontal cylinder $\mathfrak{c}'$ of radius greater or equal to $z_{2}$. Notice that, the third coordinate $x_{3}$ is bounded in $\mathcal{Z}$.  From \cite{MM},  we can assure the existence of a $[\varphi,\vec{e}_{3}]$-Bowl $\mathcal{B}$ such that $\mathcal{B}\cap\mathfrak{c}'=\varnothing$. Translating $\mathcal{B}$ in the direction $-\vec{e}_{3}$, there exists a first $l>0$ such that $\mathcal{B}-l\vec{e}_{3}$ has a tangency point $p$  of contact with $\mathcal{Z}$. Thus,  from the estimate above, we can apply the Theorem \ref{barrierprinciple} to showing that $\mathcal{B}-l\vec{e}_{3}$ is contained in $\mathfrak{c}'$ which contradicts the asymptotic behaviour of $\mathcal{B}$ given in $\cite{MM}$. It would like pointing out that the translated Bowl is a $[\varphi_{l},\vec{e}_{3}]$-minimal graph, in particular mean convex,  with $\varphi_{l}(x_{3})=\varphi (x_{3}+l)$ .
\end{proof}

\begin{center}
\includegraphics[scale=0.2]{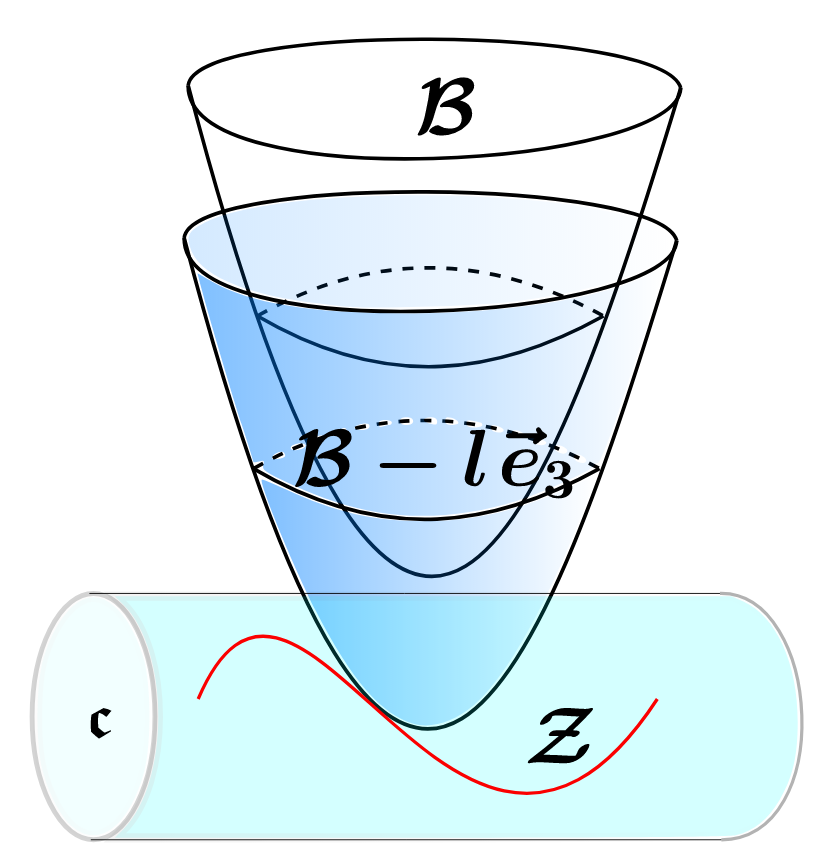}
\end{center}

From the Claims 1 and 2, the sequence $\Sigma_{n}$ have uniformly bounded area on compact subsets of $\mathbb{R}^{3}$ with respect to the Ilmanen's metric. Consequently, from the Theorem \ref{CW}, $\{\Sigma_{n}\}_{n\in\mathbb{N}}$ converges to a smooth properly embedded $[\varphi,\vec{e}_{3}]$-minimal $\Sigma_{\infty}$. Notice that the limit $\Sigma_{\infty}\neq\varnothing$ because $\{\Sigma_{n}\}$ have accumulation points due to their asymptotic behaviour. Since each $\Sigma_{n,k}^{+}(z)$ is a graph over $\mathcal{G}^{h}$ and $\Sigma_{n}$ is connected, we deduce that the multiplicity is one everywhere and so, the convergence is smooth. Moreover, observe that each component of $\Sigma_{\infty}\cap\mathcal{R}_{z}^{+}$ can be represented as the graph of a smooth function $\overline{u}_{\infty}$ which is the limit of a sequence of graphs $\overline{u}_{n}=\overline{u}(s-b_{n},t)$. Thus, $\Sigma_{\infty}$ has the same asymptotic behaviour that $\Sigma$. Finally, $\Sigma_{\infty}$ must be connected. Otherwise, there should exists a properly embedded connected component $\mathcal{S}_{\infty}$ lying inside $\mathfrak{c}$ getting to contradiction since the third coordinate of $\Sigma$ must be bounded.
\end{proof}

\begin{remark}
From the paper \cite{MM},  the hypothesis of convexity for $\varphi$ is crucial for the existence of the $[\varphi,\vec{e}_{3}]$-Bowls.
\end{remark}

\section{Proof of the main result}

In this section, we prove the uniqueness result for the $[\varphi,\vec{e}_{3}]$-catenary cylinders by their asymptotic behaviour. The main tools for proving these theorems are the moving plane method of Alexandrov \cite{Al}, the Hopf's maximum principle and a strong maximum principle for elliptic equations \cite{Hopf}. The goal will be to prove that  the angle function $\eta_{2}$ vanishes everywhere. Hence, our graph $\Sigma$ will be invariant by translations in the direction $\vec{e}_{2}$ and so, it must be a $[\varphi,\vec{e}_{3}]$-catenary cylinder. 

\begin{lemma} 
\label{l1}
Let $\varphi: ]a,+\infty[\rightarrow ]b,c[$ , $a,b\in\mathbb{R}\cup\{-\infty\}$ , $c\in\mathbb{R}\cup\{+\infty\}$ be a strictly increasing smooth function such that $e^{-\varphi}\in L^{1}(]a,+\infty[)$. If $\Sigma$ is a $[\varphi,\vec{e}_{3}]$-minimal immersion $\mathcal{C}^{\infty}$-asymptotic to $[\varphi,\vec{e}_{3}]$-catenary cylinder $\mathcal{G}^{h}$, outside a cylinder, for some $h\in ]a,+\infty[$, then $\Sigma$ is strictly contained in the slab $]-\Lambda_{h},\Lambda_{h}[\times\mathbb{R}^{2}$.
\end{lemma}
\begin{proof}
\textit{Argue by contradiction}. Let $x_{1}$ be the first coordinate of $\Sigma$ and assume that $\lambda=\sup_{\Sigma}\{x_{1}\}>\Lambda_{h}$. Consider the following subset of $\Sigma$, $$
\mathcal{S}_{h,\lambda}=\Sigma_{+}\left(\frac{\Lambda_{h}}{2}+\frac{\lambda}{2}\right) \ \ \text{with} \ \ \partial\mathcal{S}_{h,\lambda}=\Sigma\cap\Pi\left(\frac{\Lambda_{h}}{2}+\frac{\lambda}{2}\right).$$
The asymptotic behaviour of $\Sigma$ with respect to $\mathcal{G}^{h}_{+}(0)$ implies that there must exist a local maximum in the interior of $\mathcal{S}_{h,\lambda}$. From the equation \ref{laplacianox1}, we get that
$$\sup_{\mathcal{S}_{h,\lambda}}\{x_{1}\}=\sup_{\partial\mathcal{S}_{h,\lambda}}\{x_{1}\}.$$
Hence, there must be a point $p\in\partial\mathcal{S}_{h,\lambda}$ such that 
$$\frac{\Lambda_{h}}{2}+\frac{\lambda}{2}=x_{1}(p)\geq \lambda,$$
getting to a contradiction since $\lambda>\Lambda_{h}$. On the other hand, if the equality holds, comparing the vertical plane $\Pi(\Lambda_{h})$ with $\Sigma$ we have again a contradiction applying the Hopf's maximum principle \cite{Hopf}. Therefore, $\sup_{\Sigma}\{x_{1}\}<\Lambda_{h}$. 
\

Similarly, we can prove that $\inf_{\Sigma}\{x_{1}\}>-\Lambda_{h}$ thanks to the asymptotic behaviour of $\Sigma$ with respect to $\mathcal{G}^{h}_{-}(0)$.
\end{proof}

\begin{lemma}
\label{l2}
Let $\varphi: ]a,+\infty[\rightarrow ]b,c[$ , $a,b\in\mathbb{R}\cup\{-\infty\}$ , $c\in\mathbb{R}\cup\{+\infty\}$ be a strictly increasing smooth function such that $e^{-\varphi}\in L^{1}(]a,+\infty[)$. If $\Sigma$ is a $[\varphi,\vec{e}_{3}]$-minimal immersion $\mathcal{C}^{\infty}$-asymptotic to $[\varphi,\vec{e}_{3}]$-catenary cylinder $\mathcal{G}^{h}$, outside a cylinder, for some $h\in]a,+\infty[$ then $\Sigma$ is symmetric with respect to the plane $\Pi(0)$.
\end{lemma}
\begin{proof}
The main tool for proving the Lemma is the Alexandrov's method of moving planes of \cite{Al} together with the arguments of \cite{MSHS1} and \cite{MSHS2}.  Fix $t>0$ and consider the reflection of $\Sigma_{+}(t)$ with respect to $\Pi(t)$ 
$$\Sigma_{+}^{*}(t)=\{(2t-x_{1},x_{2},x_{3})\in\mathbb{R}^{3}: (x_{1},x_{2},x_{3})\in\Sigma_{+}(t)\}.$$ 
From the previous Lemma \ref{l1}, let us define the following set
\begin{equation*}
\label{conjuntoA}
\mathcal{A}=\{t\in [0,\Lambda_{h}[ : \Sigma_{+}(t) \text{ graph over } \Pi(0) \text{ and } \Sigma_{+}^{*}(t)\geq\Sigma_{-}(t)\},
\end{equation*}
where $\Sigma_{+}^{*}(t)\geq\Sigma_{-}(t)$ means that 
$$\inf [x_{1}(\pi_{t}^{-1}(x_{1},x_{2},x_{3})\cap\Sigma_{+}^{*}(t))]\geq \sup[x_{1}(\pi_{t}^{-1}(x_{1},x_{2},x_{3})\cap\Sigma_{-}(t))],$$
for any point $(x_{1},x_{2},x_{3})\in\Pi(t)$ such that $\pi_{t}^{-1}(x_{1},x_{2},x_{3})\cap\Sigma_{+}^{*}(t))\neq\varnothing$ and $\pi_{t}^{-1}(x_{1},x_{2},x_{3})\cap\Sigma_{-}(t))\neq\varnothing$ where $\pi_{t}$ is the projection onto $\Pi(t)$ given by $\pi_{t}((x_{1},x_{2},x_{3}))=(t,x_{2},x_{3})$. We will prove that $0\in\mathcal{A}$. In this case, we have that $\Sigma_{+}^{*}(0)\geq\Sigma_{-}(0)$ and by a symmetric argument, we can show that $\Sigma_{-}(0)\leq\Sigma_{+}^{*}(0)$ and so, $\Sigma_{+}^{*}(0)=\Sigma_{-}(0)$.
\begin{center}
\includegraphics[scale=0.18]{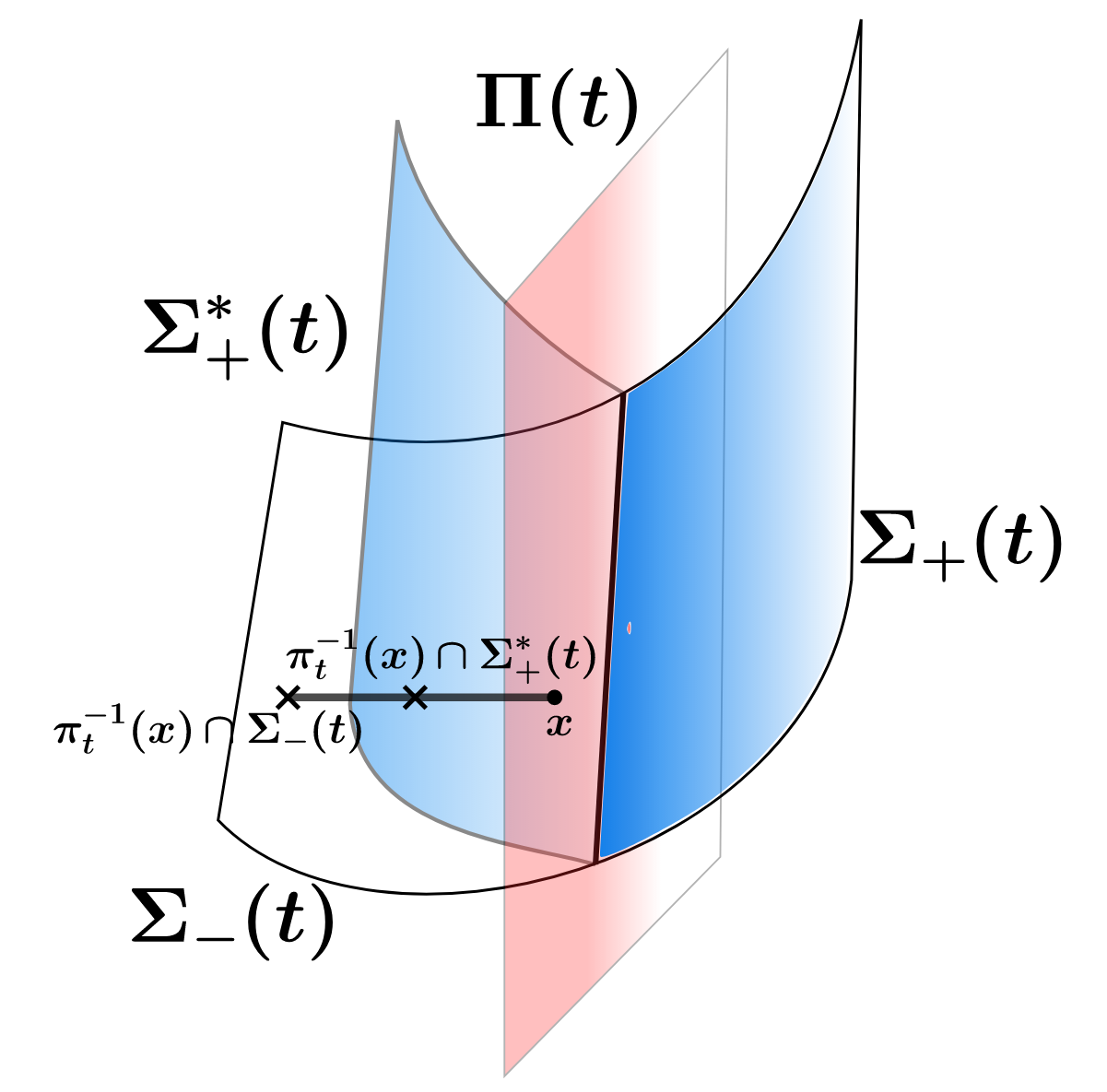}
\end{center}
Firstly, we prove that $\mathcal{A}\neq\varnothing$. From the asymptotic behaviour with respect to $\mathcal{G}^{h}_{+}(0)$, for any $\varepsilon>0$ there exists $t_{0}>0$ such that  $\Sigma_{+}(t)$ can be parametrized by
$$\widetilde{F}:T_{t,h}:=]t,\Lambda_{h}[\times\mathbb{R}\rightarrow\mathbb{R}^{3} \ \  \ \ \widetilde{F}=F+\overline{u}\, N_{F} \ \ \text{ for any } t\geq t_{0},$$
where $F$ is the parametrization of $\mathcal{G}^{h}$ given by the graph of $u$ over $T_{t}$ and $\overline{u}:T_{t,h}\rightarrow\mathbb{R}$ is a smooth function such that
\begin{equation}
\label{estimate1}
\sup_{T_{t,h}}\vert\overline{u}\vert <\varepsilon \ \ , \ \ \sup_{T_{t,h}}\vert D\overline{u}\vert<\varepsilon.
\end{equation}
Consider, 
\begin{equation}
\label{u1}
\widetilde{u}(x_{1})=\langle\widetilde{F},\vec{e}_{3}\rangle=u(x_{1})-\frac{\overline{u}(x_{1},x_{2})}{\sqrt{1+u'(x_{1})^{2}}}
\end{equation}
and define $\widetilde{u}^{*}$ the reflection with respect to $\Pi(t)$ as follows
\begin{equation}
\label{ru1}
\widetilde{u}^{*}(x_{1})=u(2t-x_{1})-\frac{\overline{u}(2t-x_{1},x_{2})}{\sqrt{1+u'(2t-x_{1})^{2}}}.
\end{equation}
From Theorem \ref{t2} and equations \eqref{estimate1},\eqref{u1} and \eqref{ru1}, the following inequality holds
\begin{equation*}
\label{ine}
\widetilde{u}^{*}(x_{1})-\widetilde{u}(x_{1})>u(2t-x_{1})-u(x_{1})-2\varepsilon
\end{equation*}
Choosing a positive constant $a$ (independent of $t_{0}$) and $t_{1}\geq t_{0}$ large enough such that $u'(x_{1})\geq \varepsilon/a$ for $a<t_{1}-x_{1}$, then the following inequalities hold
\begin{equation*}
\widetilde{u}^{*}(x_{1})-\widetilde{u}(x_{1})> 2( u'(x_{1}) a-\varepsilon )>0,
\end{equation*}
 Consequently,  from the inequality above, we get that 
\begin{align*}
\Sigma_{+}^{*}(t+a)&\cap\{(x_{1},x_{2},x_{3})\in\mathbb{R}^{3}: x_{1}\leq t\} \\
&\geq\Sigma_{-}(t+a)\cap\{(x_{1},x_{2},x_{3})\in\mathbb{R}^{3}: x_{1}\leq t\} ,
\end{align*}
for any $t\geq t_{1}$. Moreover, from the fact that $\Sigma_{+}(t)$ is a graph over $\Pi(t)$ we deduce that
\begin{align*}
\Sigma_{+}^{*}(t+a)&\cap\{(x_{1},x_{2},x_{3})\in\mathbb{R}^{3}: t\leq x_{1}\leq t+a\} \\
&\geq\Sigma_{-}(t+a)\cap\{(x_{1},x_{2},x_{3})\in\mathbb{R}^{3}: t\leq x_{1}\leq t+a\}.
\end{align*}
Hence, $[t_{1}+a,\Lambda_{h}[\subset\mathcal{A}$. Therefore, it is clear that if $s\in\mathcal{A}$ then $ ]s,\Lambda_{h}[\subset\mathcal{A}$. Now, suppose by contradiction that $s_{0}=\text{min} (\mathcal{A})>0$. Arguing as in the proof of Claim 3 in \cite{MSHS1},  there exists $\varepsilon_{1}\in ]0,\varepsilon_{0}]$ such that the surface $\Sigma_{+}(s_{0}-\varepsilon_{1})$ can be represented as a graph over the plane $\Pi(0)$. Consequently, $\Sigma_{+}(t)$ is a graph over $\Pi(0)$ for any $t\geq s_{0}-\varepsilon_{1}$. Moreover for such $\varepsilon_{1}$ and using an analogous reasoning that the Claim 3 of \cite{MSHS2}, we can prove that $\Sigma_{+}^{*}(s_{0}-\varepsilon_{1})\geq\Sigma_{-}(s_{0}-\varepsilon_{1})$ and so, $s_{0}-\varepsilon_{1}\in\mathcal{A}$ getting to contradiction.
\end{proof}

\begin{lemma}
\label{l3}
Let $\varphi: ]a,+\infty[\rightarrow ]b,c[$ , $a,b\in\mathbb{R}\cup\{-\infty\}$ , $c\in\mathbb{R}\cup\{+\infty\}$ be a convex strictly increasing smooth function such that $e^{-\varphi}\in L^{1}(]a,+\infty[)$ and $\Sigma$ be a connected $[\varphi,\vec{e}_{3}]$-minimal immersion $\mathcal{C}^{\infty}$-asymptotic to $[\varphi,\vec{e}_{3}]$-catenary cylinder $\mathcal{G}^{h}$, outside a cylinder, for some $h\in]a,+\infty]$. Consider the profile curve $\Gamma=\Sigma\cap\Pi(0)$. If the function $x_{3}\vert_{\Gamma}$ attains its global extremum on $\Gamma$, then $\Sigma$ is a $[\varphi,\vec{e}_{3}]$-catenary cylinder.
\end{lemma}

\begin{proof}
Suppose first that there exists a point $p\in\Gamma$ such that $\tau=\text{max}_{\Gamma}\{x_{3}\}=x_{3}(p)$. Observe that,
$$\partial\Sigma_{+}(0)\subset\{(x_{1},x_{2},x_{3})\in\mathbb{R}^{3}: x_{3}\leq \tau\}.$$
For any $t\in\mathbb{R}$ consider an horizontal translation of $\mathcal{G}^{\tau}_{+}(0)$ given by
$$\mathcal{G}^{\tau,t}_{+}(0)=\{(x_{1},x_{2},u(x_{1}-t))\in\mathbb{R}^{3}: x_{1}\in[t,\Lambda_{\tau}+t[\}$$
where, from the Theorem \ref{t2}, $\Lambda_{\tau}\leq\Lambda_{h}$. Define the following set
$$\mathcal{Q}=\{t\in]-\infty,0[:\mathcal{G}^{\tau,t}_{+}(0)\cap\Sigma_{+}(0)=\varnothing\}.$$
Obviously, $\mathcal{Q}\neq\varnothing$. Moreover, if $t\in\mathcal{Q}$ then $]-\infty,t[\subset\mathcal{Q}$. We claim that
$$t_{0}=\sup\{\mathcal{Q}\}=0.$$
Assume the opposite. If $t_{0}<0$ is not in $\mathcal{Q}$, then there would exist an interior point of contact since, from Theorem \ref{t2}, the boundaries of both surfaces do not touch when $t<0$. Applying the Hopf's maximum principle \cite{Hopf}, we get that $\Sigma=\mathcal{G}^{\tau,t_{0}}$ getting to contradiction with the asymptotic behaviour of $\Sigma$. Hence, we can assume that $t_{0}\in\mathcal{Q}$. In this case, because the distance between the boundaries is positive and $\mathcal{G}^{\tau,t}_{+}(0)$, $\Sigma_{+}(0)$ have different asymptotic behaviour, there exists a sequence $\{p_{n}=(p_{1,n},p_{2,n},p_{3,n})\}_{n\in\mathbb{N}}\subset\Sigma_{+}(0)$ such that $\{p_{3,n}\}$ is bounded, $\{p_{2,n}\}$ is unbounded and
$$\lim_{n\rightarrow+\infty}d(p_{n},\mathcal{G}^{\tau,t_{0}}_{+}(0))=0.$$
From Corollary \ref{conv}, we get that the sequence $\Sigma_{n}=\Sigma-(0,p_{2,n},0)$ converges smoothly, after passing to a subsequence, to a properly embedded connected $[\varphi,\vec{e}_{3}]$-minimal surface $\Sigma_{\infty}$ with the same asymptotic behaviour that $\Sigma$. However, $\Sigma_{\infty}$ and $\mathcal{G}^{\tau,t_{0}}_{+}(0)$ have an interior point of contact getting again to contradiction.
Consequently, we can assume that $t_{0}=0$. Thus, $\mathcal{G}^{\tau}_{+}(0)$ and $\Sigma_{+}(0)$ have a boundary contact at $p$. Observe that the tangent plane at $p$ of both surfaces is horizontal by Lemma \ref{l2} and so, from Hopf's maximum principle in the boundary, these surfaces coincide as we wanted. On the other hand, if there exists $q\in\Gamma$ such that $\sigma=\text{min}_{\Gamma}\{x_{3}\}=x_{3}(q)$, the same reasoning works comparing $\Sigma_{+}(0)$ with $\mathcal{G}^{\sigma,-t}_{-}(0)$ for $t\geq 0$.
\end{proof}
\begin{remark}
\label{noint}
We would like pointing out that this argument is not true when $e^{-\varphi}\notin L^{1}(]a,+\infty[)$ because we can not assure the existence of a first contact point of tangency since the inclination of the asymptotic plane of $\mathcal{G}^{h}$,given by the Theorem \ref{t2}, with respect any horizontal plane is decreasing with respect the initial condition, that is, $lim_{x_{1}\rightarrow+\infty}u'(x_{1})$ is decreasing with respect to $h$.
\end{remark}
\begin{center}
\includegraphics[scale=0.18]{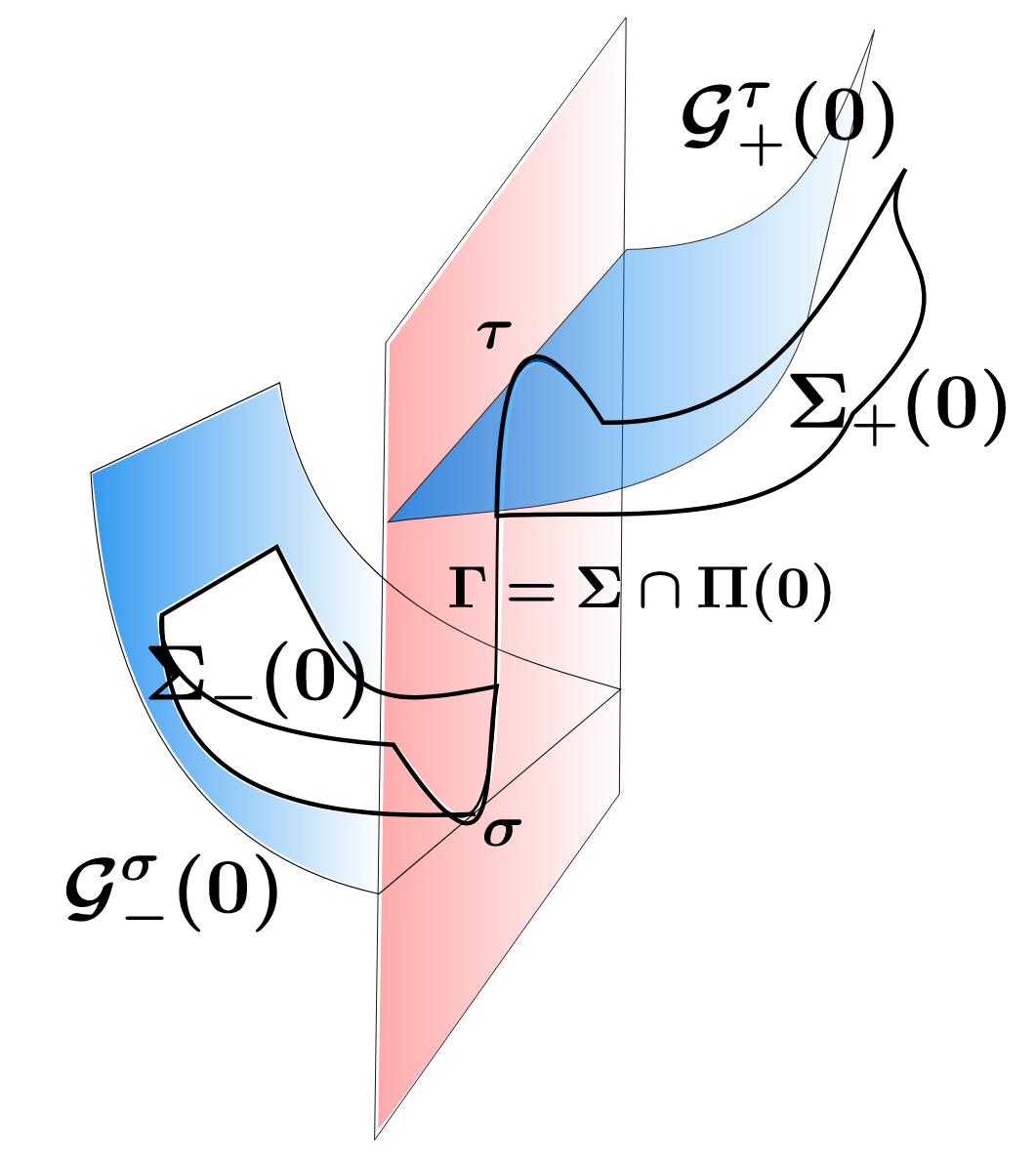}
\end{center}

\begin{proposition}
\label{comportamientocurvas}
Let $\varphi:]a,+\infty[\rightarrow ]b,c[$ , $a,b\in\mathbb{R}\cup\{-\infty\}$ , $c\in\mathbb{R}\cup\{+\infty\}$ be a convex strictly increasing diffeomorphism with $e^{-\varphi}\in L^{1} (]a,+\infty[)$ and $\Sigma$ be a connected $[\varphi,\vec{e}_{3}]$-minimal immersion $\mathcal{C}^{\infty}$-asymptotic to $[\varphi,\vec{e}_{3}]$-catenary cylinder $\mathcal{G}^{h}$, outside a cylinder, for some $h\in ]a,+\infty[$. For any sequence of points $\{(p_{1,n},p_{2,n},p_{3,n})\}$ of $\Sigma$ such that $\{p_{2,n}\}$ diverges and $\{p_{3,n}\}$ is bounded, the sequence  $\{\Sigma_{n}=\Sigma-(0,p_{2,n},0)\}_{n\in\mathbb{N}}$ converge smoothly, after subsequence, to some $[\varphi,\vec{e}_{3}]$-catenary cylinder with the same asymptotic behaviour that $\mathcal{G}^{h}$.
\end{proposition}
\begin{proof}
Up to subsequence, we can assume that $\{p_{3,n}\}$ is strictly monotonous converging to either the supremum or infimum. Now,  from the Corollary \ref{conv}, we know that $\Sigma_{n}$ converge smoothly to a connected properly embedded $[\varphi,\vec{e}_{3}]$-minimal surface $\Sigma_{\infty}$ with the same asymptotic behaviour that $\Sigma$.  Taking into account the way in which we have constructed the limit, we get that either
\begin{align}
&\text{max}_{\Pi(0)\cap\Sigma_{\infty}}\{x_{3}\}=x_{3}\vert_{\Sigma_{\infty}}((0,0, p_{3,\infty}))\label{a1} \text{ or, }\\
&\text{min}_{\Pi(0)\cap\Sigma_{\infty}}\{x_{3}\}=x_{3}\vert_{\Sigma_{\infty}}((0,0,p_{3,\infty}))\label{b1}.
\end{align}
Consequenly, from the Lemma \ref{l3} and \eqref{a1} or \eqref{b1}, we prove that $\Sigma_{\infty}$ coincides with some $[\varphi,\vec{e}_{3}]$-catenary cylinder.
\end{proof}

\noindent{\large PROOF OF THEOREM \ref{uni}}

\begin{proof}
Let $\{a_{n}\}_{n\in\mathbb{N}}, \{ b_{n}\}_{n\in\mathbb{N}}$ be a strictly increasing sequences of positive real numbers.  Consider a compact exhaustion $\{\Lambda_{n}\}_{n\in\mathbb{N}}$ whose boundary consists in the following four curves,
\begin{align*}
&\Lambda_{1,n}=\{(x_{1},x_{2},x_{3})\in\Sigma: x_{1}>0,\ \ -a_{n}\leq x_{2}\leq b_{n}, \ \ x_{3}=n\}, \\
&\Lambda_{2,n}=\{(x_{1},x_{2},x_{3})\in\Sigma: x_{1}<0,\ \ -a_{n}\leq x_{2}\leq b_{n}, \ \ x_{3}=n\}, \\
&\Lambda_{3,n}=\{(x_{1},x_{2},x_{3})\in\Sigma: x_{2}=-a_{n} \ \ x_{3}\leq n\}, \\
&\Lambda_{4,n}=\{(x_{1},x_{2},x_{3})\in\Sigma: x_{2}=b_{n} \ \ x_{3}\leq n\}.
\end{align*}
\begin{center}
\includegraphics[scale=0.16]{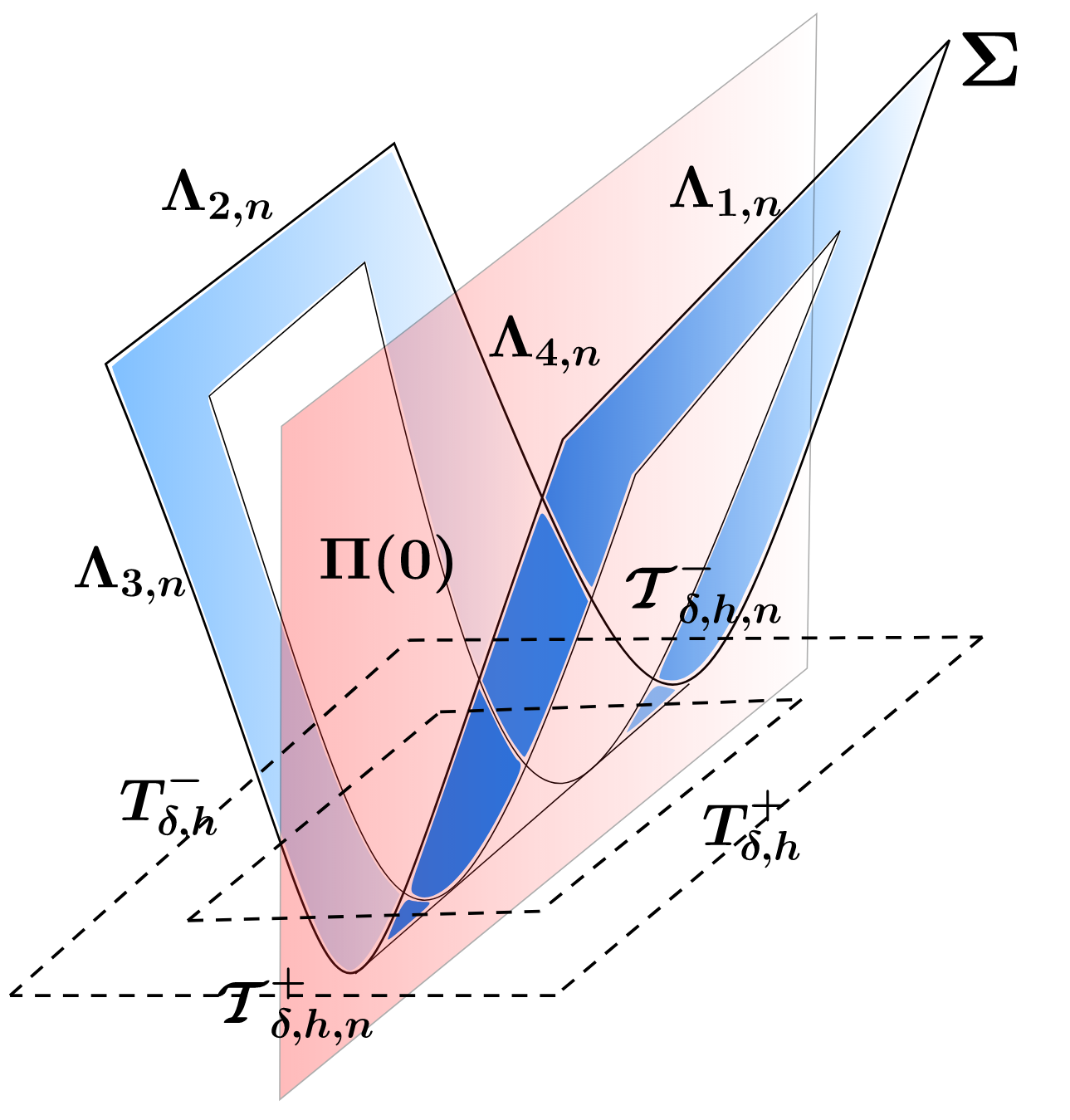}
\end{center}
From the asymptotic behaviour given by the definition \ref{convergenciaG}, fix  $n_{1}$ large enough and consider an arbitrary $n\geq n_{1}$, for any $\varepsilon>0$ there exists  $\delta>0$ (depending only of $\varepsilon$) such that a neighborhood of $\Lambda_{1,n}$ can parametrized as a graph over $\mathcal{G}^{h}_{+}(0)$ by $$\widetilde{F}:T_{\delta,h}^{+}\rightarrow\mathbb{R}^{3} \ \ , \ \  \widetilde{F}=F+\overline{u}\, N_{F},$$
where $F(x_{1},x_{2})=(x_{1},x_{2},u(x_{1}))$ parametrizes $\mathcal{G}^{h}$ on $T_{\delta,h}^{+}$ with $u$ solution of \eqref{htrans}, $\overline{u}:T_{\delta,h}^{+}\rightarrow\mathbb{R}$ is a smooth function such that
$$\sup_{T_{\delta,h}^{+}}\vert\overline{u}\vert <\varepsilon \ \ , \ \ \sup_{T_{\delta,h}^{+}}\vert D^{j}\overline{u}\vert<\varepsilon, \text{ for any } j\in\mathbb{N}.$$
and $N_{F}$ is the unit normal of $\mathcal{G}^{h}$. From the equation \eqref{HFF}, we can write
\begin{equation}
\label{funcionh}
\frac{\eta_{2}}{\eta_{3}}=\overline{u}_{x_{2}}\sqrt{1+u'^{2}}\left(\frac{1+\overline{u}\left(\frac{\dot{\varphi}}{\sqrt{1+u'^{2}}}\right)}{1+\overline{u}_{x_{1}}\frac{u'}{\sqrt{1+u'^{2}}}+\overline{u}\left(\frac{\dot{\varphi}}{\sqrt{1+u'^{2}}}\right)}\right) \ \ \text{on } T_{\delta,h}^{+}.
\end{equation}
Let us examine at first the behaviour of $\eta_{2}/\eta_{3}$ along $\Lambda_{1,n}$. Because for any fixed $x_{2}$, we have that
$\lim_{x_{1}\rightarrow\Lambda_{h}}\vert\overline{u}\vert=\lim_{x_{1}\rightarrow\Lambda_{h}}\vert D\overline{u}\vert=0,$
then
\begin{equation}
\label{cotau}
\vert\overline{u}_{x_{2}}(x_{1},x_{2})\vert=\bigg\vert -\int_{x_{1}}^{\Lambda_{h}}\overline{u}_{x_{1}x_{2}}(s,x_{2})\, d s\bigg\vert\leq (\Lambda_{h}-x_{1})\varepsilon.
\end{equation}
Moreover, from the Theorem \ref{t2}, the L'H\^opital rule and assumptions of Theorem A, the following limits exists
\begin{equation}
\label{limites}
\lim_{x_{1}\rightarrow\Lambda_{h}}\frac{\dot{\varphi}}{\sqrt{1+u'^{2}}}\ \ , \ \ \lim_{x_{1}\rightarrow\Lambda_{h}}(\Lambda_{h}-x_{1})\sqrt{1+u'^{2}(x_{1})}.
\end{equation}
Consequently, from \eqref{cotau} and \eqref{limites}, there exists $n_{2}\geq n_{1}$ large enough such that for any $n\geq n_{2}$, the following estimates holds for the equation \eqref{funcionh},
\begin{equation}
\label{d4}
 \text{sup}_{\Lambda_{1,n}}\vert \eta_{2}/\eta_{3}\vert<\varepsilon.
\end{equation}
Moreover, from the symmetry of \ref{l2}, the previous inequality \eqref{d4} holds in $\Lambda_{2,n}$. On the other hand, from the Proposition \ref{comportamientocurvas} , we can argue analogously with the curves $\Lambda_{3,n}$ and $\Lambda_{4,n}$ because both curves are $\mathcal{C}^{\infty}$-asymptotic to some $\mathcal{G}^{h'}$ such that $\mathcal{G}^{h'}$ has the same asymptotic behaviour that $\mathcal{G}^{h}$, that is, $\Lambda_{h'}=\Lambda_{h}$. Hence, there exists  $n_{3}\geq n_{2}$ large enough such that for any $n\geq n_{3}$ a neighborhood of $\Lambda_{3,n}$ can be parametrized as a graph over $\mathcal{G}^{h'}$ by
$$\widetilde{F}_{n}:\mathcal{T}_{\delta,h,n}^{+}\rightarrow\mathbb{R}^{3} \ \ \ \ \widetilde{F}_{n}=F+\overline{u_{n}}N_{F},$$ 
where $\mathcal{T}_{\delta,h,n}^{+}=]-\Lambda_{h'}+\delta,\Lambda_{h'}-\delta[\times ]m_{1,n},m_{2,n}[\rightarrow\mathbb{R}$ is a smooth function, $\delta>0$ only depends of $n_{3}$, $\{m_{1,n}\}_{n\in\mathbb{N}},\{m_{2,n}\}_{n\in\mathbb{N}}$ are strictly monotonous sequences with $m_{1,n}<m_{2,n}$ and each $\overline{u_{n}}:\mathcal{T}_{\delta,h,n}^{+}\rightarrow\mathbb{R}$ is a smooth function satisfying the following inequalities,
$$\text{sup}_{\mathcal{T}_{\delta,h,n}^{+}}\vert\overline{u_{n}}\vert<\varepsilon, \ \ \text{sup}_{\mathcal{T}_{\delta,h,n}^{+}}\vert D^{j}\overline{u_{n}}\vert<\varepsilon, \ \ \text{for any } j\in\mathbb{N}.$$
Notice that, the existence of these sequence of functions is guaranteed by the convergence given in the Theorem \ref{CW}. In these case, $x_{1}$ is not tending to $\pm\Lambda_{h}$ and from the Theorem \ref{t2}, $u'$ is bounded. Hence, the same argument works as above because the previous limits \eqref{limites} exist for any divergence sequence of points and so, the inequality \eqref{d4} holds over $\Lambda_{3,n}$.  Analogously, taking $n_{4}\geq n_{3}$ large enough and parametrizing a neighborhood of $\Lambda_{4,n}$ as a graph over $\mathcal{G}^{h'}$  by $\overline{v}_{n}:\mathcal{T}_{\delta,h,n}^{-}=]-\Lambda_{h'}+\delta,\Lambda_{h'}-\delta[\times ]-m_{2,n},-m_{1,n}[ \rightarrow\mathbb{R}$ with
$$\text{sup}_{\mathcal{T}_{\delta,h,n}^{-}}\vert\overline{v_{n}}\vert<\varepsilon, \ \ \text{sup}_{\mathcal{T}_{\delta,h,n}^{-}}\vert D^{j}\overline{v_{n}}\vert<\varepsilon \ \ \text{for any } j\in\mathbb{N},$$
and the inequality \eqref{d4} also holds over $\Lambda_{4,n}$. Consequently, we get that the function $\eta_{2}/\eta_{3}$ tends to zero as $p\rightarrow \infty$. In particular, there exists an interior point where the function $\eta_{2}/\eta_{3}$ attains either a local minimum or a local maximum. From the equation \eqref{LAP2} and the convexity of $\varphi$, we deduce that $\eta_{2}$ vanishes everywhere and then, $\Sigma$ is invariant under translations in the direction $\vec{e}_{2}$ as we wanted.
\end{proof}

\begin{remark}
Notice that if $\Lambda_{\lambda}$ is strictly decreasing with respect to inital data $\lambda$, then we can assure that $\Sigma$ coincides with $\mathcal{G}^{h}$.
\end{remark}

\section{Concluding remarks} 
\begin{itemize}
\item As a direct consequence of the proof of Theorem \ref{uni}, we can change the hypothesis of to be graph in the Theorem \ref{uni} by the following conditions. Suppose that $\varphi$ is a strictly increasing convex smooth function with at most a linear growth such that $\dddot{\varphi}\leq 0$  and assume that $\Sigma$ is a connected complete mean convex properly embedded $[\varphi,\vec{e}_{3}]$-minimal immersion with locally bounded genus $\mathcal{C}^{\infty}$-asymptotic to some $[\varphi,\vec{e}_{3}]$-catenary cylinder, outside a cylinder. Notice that, $\dot{\varphi}$ and $\ddot{\varphi}$ are bounded. Thus, $\Sigma$ is closed to the $[\varphi,\vec{e}_{3}]$-catenary cylinder in $C^{2}$-topology. In particular the Gauss curvature of $\Sigma$ is bounded. From the section 4. and the Theorem B of \cite{MMJ}, $\Sigma$ has locally uniformly bounded intrinsic area and $K\geq 0$ everywhere. If we prove that $\eta_{2}=0$ everywhere, then there exists a point where the Gauss curvature vanishes. Consequently, from the Theorems 2.5 and 3.7 of \cite{MM}, $\Sigma$ is flat and it must be coincide with some $[\varphi,\vec{e}_{3}]$-catenary cylinder.

\item On the other hand, an interesting question is whether a complete properly embedded $[\varphi,\vec{e}_{3}]$-minimal surface with locally bounded genus $\mathcal{C}^{\infty}$-asymptotic, outside a cylinder, to some $[\varphi,\vec{e}_{3}]$-catenary cylinder is a vertical graph with $\varphi$ under the conditions of  theorem \ref{uni}. In such case, we  can generalize the result of  F. Mart\'in, J. P\'erez-Garc\'ia, A. Savas-Halilaj and K. Smoczyk \cite{MSHS1}.

\item Finally, from the remark \ref{noint}, it would be interesting to prove the veracity of the Theorem \ref{uni} when $\varphi$ is strictly increasing convex smooth function but $e^{-\varphi}$ is not integrable.
\end{itemize}


\begin{thebibliography}{10}

\bibitem {Al} Alexandrov, A.D.: \textit{Uniqueness theorems for surfaces in the large.} Vestnik. Leninger. Univ. Math. \textbf{11}, 5-17 (1956).

\bibitem{D}U. Dierkes: \textit{Singular Minimal Surfaces},  S. Hildebrandt  et. al (eds.), Geometric Analysis and Nonlinear Partial Differential Equations, (2003) 177-193.

\bibitem{D1} U. Dierkes, G. Huisken: \textit{ The N-dimensional analogue of the catenary: Prescribed area}. J. Jost (ed) Calculus of Variations and Geometric Analysis. International Press, (1996) 1-13.


\bibitem{HIMW} D. Hoffman, T. Ilmanen, F. Mart\'in and B. White: \textit{Notes on translating solitons for mean curvature flow}, prePrint, arXiv: 1901.09101v1, (2019).

\bibitem{HIMW2} D. Hoffman, T. Ilmanen, F. Mart\'in and B. White: \textit{Graphical translators for mean curvature flow}. Calc. Var. Partial Differential Equations 58 (2019), no. 4, Paper No. 117, 29 pp.

\bibitem{Hopf} Gilbarg. D., Trudinger. N.S: \textit{Elliptic Partial Differential Equations of Second Order}, Classics in Mathematics. Springer, Berlin (2001). (Reprint of the 1998 Edition).

\bibitem{Ilm94} T. Ilmanen:\textit{ Elliptic regularization and partial regularity for motion by mean curvature}. Men. Amer. Math. Soc, 108 (1994).

\bibitem{RL} R. L\'opez :\textit{ Constant Mean Curvature Surfaces with Boundar }. Springer Monographs in Mathematics, (2013).


\bibitem{MSHS1} F.Mart\'in, J. P\'erez-Garc\'ia, A.Savas-Halilaj and K.Smoczyk: \textit{A characterization of the grim reaper cylinder}, Journal fur die reine und angewandte Mathematik (2016), 1-28.

\bibitem{MSHS2} Mart\'in.F, A.Savas-Halilaj and Smoczyk,K: \textit{ On the topology of translating solitons of the mean curvature flow}, Cal. Var. \textbf{54} (2015):2853-2882.


\bibitem{MM} A. Mart\'inez. and A. L. Mart\'inez Trivi\~no: \textit{Equilibrium of Surfaces in a Vertical Force Field}. Preprint: arXiv: 1910.07795, (2019).

\bibitem{MMJ}  A. Mart\'inez, A.L. Mart\'inez Trivi\~no and J.P. dos Santos: \textit{Mean convex properly embedded $[\varphi,\vec{e}_{3}]$-minimal surfaces in $\mathbb{R}^{3}$}. arXiv: 2011.15029v1 (2020).

\bibitem{Poisson} S.D.Poisson. \textit{Sur les surfaces elastique.} \ Men. \ CL. \ Sci. \ Math. \ Phys. \ Inst. \ Frace, deux, 167-225 (1975).

\bibitem{SX}J. Spruck, and L. Xiao:  \textit{Complete translating solitons to the mean curvature flow in $\mathbb{R}^3$ with nonnegative mean curvature}, Amer. J. Math. 142 (2020), no. 3, 993--1015. 

\bibitem{Wang} X. Wang:\textit{ Convex solutions to the mean curvature flow of mean-convex sets}, J. Amer. Math. Soc, pages 123-138, 16 (2003).

\bibitem{W3} B. White: \emph{ Controlling area blow-up in minimal or bounded mean curvature varities }. J.Differential Geom. Volume 102, no.3 (2016), 501-535.

\bibitem{W2} B. White: \emph{On the compactness theorem for embedded minimal surfaces in 3-manifolds with locally bounded area and genus} Comm. Anal. Geom. 26 (2018), no. 3, 659--678.

\end{thebibliography}
\end{document}